\documentclass[a4,12pt]{amsart}
\usepackage{amssymb,amstext,amsmath,amscd,amsthm,amsfonts,enumerate,latexsym}
\usepackage{color}
\usepackage[dvipdfmx]{graphicx}
\usepackage[all]{xy}
\usepackage{stmaryrd} 

\usepackage[dvipdfmx]{hyperref}

\theoremstyle{plain}
\newtheorem{thm}{Theorem}[section]

\newtheorem*{thm*}{Theorem}
\newtheorem*{cor*}{Corollary}

\newtheorem{prop}[thm]{Proposition}

\newtheorem{lem}[thm]{Lemma}
\newtheorem{cor}[thm]{Corollary}

\newtheorem*{claim*}{Claim}

\theoremstyle{definition}
\newtheorem{defn}[thm]{Definition}

\newtheorem{ex}[thm]{Example}

\newtheorem{ques}[thm]{Question}

\newtheorem{setting}[thm]{Setting}

\theoremstyle{remark}
\newtheorem{rem}[thm]{Remark}

\numberwithin{equation}{thm}

\newtheorem*{ac}{Acknowledgments}
\newtheorem*{sdc}{Statements and Declarations}

\def\Ext{\operatorname{Ext}}

\def\Im{\operatorname{Im}}

\def\Hom{\operatorname{Hom}}

\def\Max{\operatorname{Max}}

\def\End{\mathrm{End}}

\def\mod{\mathrm{mod}}

\def\Coker{\mathrm{Coker}}

\def\e{\mathrm{e}}
\def\m{\mathfrak m}

\def\p{\mathfrak p}

\newcommand{\rma}{\mathrm{a}}

\newcommand{\rmc}{\mathrm{c}}

\newcommand{\rme}{\mathrm{e}}
\newcommand{\rmf}{\mathrm{f}}

\newcommand{\rmr}{\mathrm{r}}

\newcommand{\rmJ}{\mathrm{J}}
\newcommand{\rmK}{\mathrm{K}}

\newcommand{\rmQ}{\mathrm{Q}}

\newcommand{\calF}{\mathcal{F}}

\newcommand{\calP}{\mathcal{P}}

\newcommand{\calW}{\mathcal{W}}
\newcommand{\calX}{\mathcal{X}}
\newcommand{\calY}{\mathcal{Y}}
\newcommand{\calZ}{\mathcal{Z}}

\newcommand{\fka}{\mathfrak{a}}

\newcommand{\fkc}{\mathfrak{c}}

\newcommand{\fkp}{\mathfrak{p}}

\newcommand{\mapright}[1]{%
\smash{\mathop{%
\hbox to 1cm{\rightarrowfill}}\limits^{#1}}}

\newcommand{\mapleft}[1]{%
\smash{\mathop{%
\hbox to 1cm{\leftarrowfill}}\limits_{#1}}}

\def\depth{\operatorname{depth}}

\def\Ass{\operatorname{Ass}}

\def\Spec{\operatorname{Spec}}
\def\tr{\operatorname{tr}}
\def\Syz{\mathrm{Syz}}

\def\OCM{\operatorname{\rm \Omega \rm{CM}}}
\def\CM{\operatorname{\rm{CM}}}
\def\Ref{\operatorname{\rm{Ref}}}
\def\indCM{\operatorname{\rm{ind\hspace{0.1em} CM}}}
\def\indOCM{\operatorname{\rm{ind\hspace{0.1em}\Omega CM}}}
\def\ind{\operatorname{\rm{ind}}}

\title[Reflexive modules over the rings $I:I$]{Reflexive modules over the endomorphism algebras of reflexive trace ideals}

\author[Naoki Endo]{Naoki Endo}
\address{School of Political Science and Economics, Meiji University, 1-9-1 Eifuku, Suginami-ku, Tokyo 168-8555, Japan}
\email{endo@meiji.ac.jp}
\urladdr{https://www.isc.meiji.ac.jp/~endo/}

\author[Shiro Goto]{Shiro Goto}
\address{Department of Mathematics, School of Science and Technology, Meiji University, 1-1-1 Higashi-mita, Tama-ku, Kawasaki 214-8571, Japan}
\email{shirogoto@gmail.com}

\thanks{2020 {\em Mathematics Subject Classification.} 13H10, 13A15, 13C14.}
\thanks{{\em Key words and phrases.} reflexive module, torsionfree module, endomorphism algebra of an ideal}
\thanks{The first author was partially supported by JSPS Grant-in-Aid for Young Scientists 20K14299. The second author was partially supported by JSPS Grant-in-Aid for Scientific Research (C) 21K03211. }
\thanks{The second author passed away on July 26, 2022.}

\begin{document}

\maketitle

\setlength{\baselineskip} {14.9pt}

\begin{abstract}
In the present paper we investigate reflexive modules over the endomorphism algebras of reflexive trace ideals in a one-dimensional Cohen-Macaulay local ring. The main theorem generalizes both of the results of S. Goto, N. Matsuoka, and T. T. Phuong (\cite[Theorem 5.1]{GMP}) and T. Kobayashi (\cite[Theorem 1.3]{K}) concerning the endomorphism algebra of its maximal ideal. We also explore the question of when the category of reflexive modules is of finite type, i.e., the base ring has only finitely many isomorphism classes of indecomposable reflexive modules. 
We show that, if the category is of finite type, the ring is analytically unramified and has only finitely many Ulrich ideals. As a consequence, there are only finitely many Ulrich ideals are contained in Arf local rings once the normalization is a local ring.
\end{abstract}



\section{Introduction}\label{intro}

This paper aims at, in one-dimensional Cohen-Macaulay local rings,  investigating the category of reflexive modules over the endomorphism algebras of reflexive trace ideals.  

Let $R$ be a commutative Noetherian ring. For an $R$-module $M$, we have a canonical homomorphism 
$
\varphi: M \to M^{**}
$ 
defined by ${\big[ \varphi(x) \big](f) = f(x) }$ for each ${f \in M^{*} }$ and $x \in M$, where $(-)^{*} =\Hom_{R}(-, R)$ denotes the $R$-dual functor. Following the terminology of H. Bass \cite[page 476]{Bass}, we say that $M$ is {\it reflexive} if $\varphi$ is bijective, {\it torsionless} if $\varphi$ is injective.
Torsionless modules are {\it torsionfree}, i.e., there is no nonzero torsion elements, and the converse holds if the total ring of fractions $\rmQ(R)$ of $R$ is Gorenstein (\cite[Theorem (A.1)]{V1}); equivalently, the local ring $R_{\p}$ is Gorenstein for every $\p \in \Ass R$. 
Simplest examples of reflexive modules are finite free modules. 
The notion of reflexivity of modules, in general, appears not only commutative algebra but in diverse branches of mathematics, and it often plays a crucial role in many situations. 
Among them, our interest is what the relative size of the category of reflexive, as well as torsionfree, modules reflects the singularities of base rings. A one-dimensional Cohen-Macaulay local ring is Gorenstein if and only if all finite reflexive modules are torsionfree; the regularity of the ring is the same as having all finite torsionfree modules are free. 
In addition, provided that the Noetherian ring $R$ satisfies Serre's $(S_2)$ condition and $\rmQ(R)$ is Gorenstien, a finitely generated $R$-module $M$ is reflexive if and only if $M_{\p}$ is reflexive as an $R_{\p}$-module for every $\p \in \Spec R$ with $\dim R_{\p} = 1$ and $M$ satisfies $(S_2)$, where a finitely generated $R$-module $N$ satisfies {\it Serre's $(S_2)$ condition} if $\depth_{R_{\p}}N_\p \ge \inf \{2, \dim R_{\p}\}$ for every $\p \in \Spec R$. 
This observation suggests us that the importance of analysis in the one-dimensional case for reflexive modules.


To explain our aim and motivation more precisely, let $(R, \m)$ be a Cohen-Macaulay local ring with $\dim R=1$ satisfying $(S_2)$. We further assume $\rmQ(R)$ is a Gorenstein ring. Note that this assumptions is automatically satisfied if $R$ is reduced. 
 Let $A$ be a {\it birational module-finite extension} of $R$, namely, $A$ is an intermediate ring between $R$ and $\rmQ(R)$ which is finitely generated as an $R$-module. A finitely generated $A$-module $M$ is called  {\it maximal Cohen-Macaulay} (abbr. MCM) if $\depth_{A_\p}M_\p \ge \dim A_{\p}$ holds for every $\p \in \Spec A$. As $\dim A=1$, MCM modules coincide with torsionfree. 
We define the categories below which we will frequently use throughout this paper: 
\begin{itemize}
\item $\mod\, A$ the category of finitely generated $A$-modules,
\item $\CM(A)$ the full subcategory of $\mod\, A$ consisting of MCM $A$-modules, and
\item $\Ref(A)$ the full subcategory of $\mod\, A$ consisting of reflexive $A$-modules. 
\end{itemize}
Then $\Ref(A) \subseteq \CM(A)$. Since reflexive modules $M \in \mod\, A$ appear in the exact sequence $0 \to M \to F_1 \to F_2$ with finite free modules $F_1$ and $F_2$, we have $\Ref(A) = \OCM(A)$, the full subcategory of $\mod\, A$ consisting of first syzygies of MCM $A$-modules. This means $\OCM(A)$ consists of all $M \in \mod\, A$ such that there exists an exact sequence $0 \to M \to F \to X \to 0$ of $A$-modules, where $F$ is a free $A$-module and $X \in \CM(A)$. 
One can show that $A$ is a Gorenstein ring if and only if the equality $\OCM(A) = \CM(A)$ holds. 

The present research traces back to a classical question of when the endomorphism algebra $E=\End_R(\m)\, (\cong \m:\m)$, which forms a commutative ring, of $\m$ is a Gorenstein ring. This question is originally advocated by V. Barucci and R. Fr\"{o}berg (\cite[Proposition 25]{BF}) in the case where the ring $R$ is  {\it analytically unramified}, i.e., the $\m$-adic completion $\widehat{R}$ is reduced, and then S. Goto, N. Matsuoka, and T. T. Phoung settled the above question with full generality. They proved in  \cite[Theorem 5.1]{GMP}, for an arbitrary Cohen-Macaulay local ring $R$ and of dimension one,  that $E=\End_R(\m)$ is Gorenstein, that is, $\OCM(E) = \CM(E)$, if and only if $R$ is {\it almost Gorenstein} in the sense of \cite{GMP} and its maximal ideal $\m$ is stable, i.e., $\m^2 = a\m$ for some $a \in \m$.  
Later, in \cite{K}, T. Kobayashi introduced the full subcategory $\OCM_{\calP}(R)$ of $\mod\, R$ consisting of MCM $R$-modules without free summands, and proved that one has $\OCM(E) \subseteq \OCM_{\calP}(R) \subseteq \CM(E)$ as subcategories of $\mod\, R$, while the equality $\OCM(E) = \OCM_{\calP}(R)$ holds if and only if $\m$ is stable; the equality $\OCM_{\calP}(R) = \CM(E)$ holds if and only if $R$ is an almost Gorenstein ring.



While tracking the development of almost Gorenstein rings, the authors often encounter non-almost Gorenstein rings which are still have good properties; see e.g., \cite{CCGT, CGKM, EGI, GIT, GK}. Thus the theory of reflexive modules might develop more in relation to the singularities of rings and one still needs to pursue the works \cite{BF, GMP, K} in the literature. In this direction we consider the following question.

\begin{ques}
What happens if we take the birational module-finite extension $A$ to be the endomorphism algebra $\End_R(I)$ of an (not nesessarily maximal) ideal $I$ of $R$?
\end{ques}


%

The category $\OCM(A)$, which can be defined for higher dimensional rings $A$, has been deeply connected to the theory of special Cohen-Macaulay modules in the sense of \cite{Wunram}. If $A$ is a two-dimensional complete local normal domain, all MCM $A$-modules are exactly reflexive, and for $M \in \CM(A)$, the $A$-module $M$ is special Cohen-Macaulay if and only if $M^* \in \OCM(A)$; see \cite[Theorem 2.7]{IW}. In addition, it is proved in \cite[Corollary 3.3]{DITV} that, for a two-dimensional excellent henselian local normal domain $A$ with algebraically closed field, $\OCM(A)$ is of finite type if and only if the ring $A$ has a rational singularity; equivalently, $A$ has a non-commutative resolution. Here, a full subcategory $\calX$ of $\mod\, A$ is called {\it of finite type} if $\calX = {\rm add}_AM$ for some $M \in \mod\, A$, where ${\rm add}_AM$ is the full subcategory of $\mod\, A$ consisting of direct summands of finite direct sums of copies of $M$. We then naturally attain the following.


\begin{ques}
When is $\OCM(R)$ of finite type for a one-dimensional ring $R$?
\end{ques}

By \cite[Corollary 1.4]{K}, provided that $R$ is an almost Gorenstein ring, the category $\OCM(R)$ is of finite type if and only if so is $\CM(\End_R(\m))$. However, it remains unclear what happens if we do not assume the ring is almost Gorenstein, which we will clarify in the present paper.

We now state our results explaining how this paper is organized. In Section \ref{sec2}, we summarize some basic properties of the algebra $R^I = \bigcup_{n\geq 0}\left[I^n: I^n\right]$ of $R$ at an ideal $I$ and trace ideals, which play an important role in our argument.
Section \ref{sec3} is devoted to prove Theorem \ref{main}, a main result of this paper. This provides a complete generalization of \cite[Theorem 5.1]{GMP} and \cite[Theorem 1.3]{K} as well.
In Section \ref{sec4}, we explore the question of when the category $\OCM(R)$ is of finite type. 
Suppose that $R$ admits the canonical module $\rmK_R$ and the existence of a {\it fractional canonical ideal} $K$ of $R$, i.e., $K$ is an $R$-submodule of $\rmQ(R)$ such that $R \subseteq K \subseteq \overline{R}$ and $K \cong \rmK_R$, where $\overline{R}$ denotes the integral closure of $R$ in $\rmQ(R)$.
Theorem \ref{4.2} claims that, if $R$ is a {\it generalized Gorenstein ring}, defined in \cite{GK}, with minimal multiplicity, then $\OCM(R)$ is of finite type if and only if so is $\CM(R[K])$. We explore the examples as well.
In Section \ref{sec5}, we study the relation between the finite typeness of $\OCM(R)$ and the set of Ulrich ideals, a special kind of $\m$-primary ideals in $R$.
 Although we restrict to one-dimensional rings, Theorem \ref{4.10} gives a generalization of \cite[Theorem 7.8]{GOTWY}. In Section \ref{sec6}, we furthermore extend the classical result that every one-dimensional Cohen-Macaulay local ring of finite CM-representation type is analytically unramified. Theorem \ref{5.4} states the ring $R$ is analytically unramified if $\OCM(R)$ is of finite type. 
The converse holds if $R$ is an Arf ring and the normalization $\overline{R}$ is a local ring, which is also pointed out by the recent papers \cite{D, DL, IK}. As a consequence, we show that there are only finitely many Ulrich ideals are contained in Arf local rings once the normalization is a local ring.

\section{Preliminaries}\label{sec2}

Let $R$ be an arbitrary commutative ring. An ideal $I$ of $R$ is called {\it regular} if it contains a non-zerodivisor on $R$. We denote by $\rmQ(R)$ the total ring of fractions of $R$.
For $R$-submodules $X$ and $Y$ of $\rmQ(R)$, let $X:Y = \{a \in \rmQ(R) \mid aY \subseteq X\}$. If we consider ideals $I, J$ of $R$, we set $I:_RJ =\{a \in R \mid aJ \subseteq I\}$; hence $I:_RJ = (I:J) \cap R$.
A {\it fractional ideal} $I$ of $R$ is a finitely generated $R$-submodule of $\rmQ(R)$ satisfying $\rmQ(R)\cdot I = \rmQ(R)$. For fractional ideals $I$ and $J$, we have a natural identification $I:J \cong \Hom_R(J, I)$. In particular, the endomorphism algebra $\End_R(I)\cong I:I$ of a fractional ideal $I$ forms a commutative and birational module-finite extension of $R$. 
Note that a fractional ideal $M$ of $R$ is reflexive if and only if the equality $R:(R:M) = M$ holds (\cite[Proposition 2.4]{KT}).


For each regular ideal $I$ of $R$, there is a filtration of endomorphism algebras as follows:
$$
R \subseteq I:I \subseteq I^2: I^2 \subseteq \cdots \subseteq I^n:I^n \subseteq \cdots \subseteq\rmQ(R).
$$
Set $R^I = \bigcup_{n\geq 0}\left[I^n: I^n\right]$. The ring $R^I$ is a birational extension of $R$, and it coincides with the blow-up of $R$ at $I$ when $R$ is Noetherian and of dimension one.  For each $n > 0$, the ideal $I^n$ is regular and $R^I = R^{I^n}$. Let $\overline{R}$ denote the integral closure of $R$ in $\rmQ(R)$. Hence, for a Cohen-Macaulay local ring $R$ of dimension one, all the blow-ups $R^I$ of $R$ at regular ideals $I$ are finitely generated $R$-modules and $R \subseteq R^I \subseteq \overline{R}$, because there exists an integer $n >0$ such that $I^n$ contains a principal reduction; see \cite[Proof of Proposition 1.1]{Lipman}.
Note that, if $a \in I$ is a reduction of $I$, i.e., $I^{r+1} = a I^r$ for some $r\ge0$, then one has
$$
R^I = R\left[\frac{I}{a}\right] = \frac{I^r}{a^r}, \ \  \text{ where } \ \frac{I}{a} = \left\{\frac{x}{a} ~\middle|~ x \in I\right\} \subseteq \rmQ(R)
$$
while $R^I = I^n:I^n$ for every $n \ge r$. A regular ideal $I$ of $R$ is called {\it stable} if $R^I=I:I$, or equivalently, $I^2 = aI$ for some $a \in I$; see \cite[Lemma 1.8]{Lipman}. So, the ring $R^I$ coincides with $\End_R(I)$ for stable ideals $I$. 
When $R$ is a one-dimensional Cohen-Macaulay local ring, the stability of its maximal ideal $\m$ is equivalent to the ring has {\it minimal multiplicity}, i.e., the Hilbert-Samuel multiplicity is equal to the embedding dimension.


We recall the definition and basic properties of trace ideals. 
Let $M, X$ be $R$-modules and consider the homomorphism of $R$-modules
$$
\tau: \Hom_R(M,X) \otimes_RM \to X
$$ 
defined by $\tau(f \otimes m)=f(m)$ for each $f \in \Hom_R(M,X)$ and $m \in M$. Set $\tr_X(M)= \Im \tau$, which forms an $R$-submodule of $X$, and call it {\it the trace module of $M$ in $X$}.  


\begin{prop}[{\cite[Lemma 2.3]{L2}, \cite[Corollary 2.2]{GIK}}]\label{1.10}
Let $I$ be an ideal of $R$. Then the following conditions are equivalent.
\begin{enumerate}[$(1)$]
\item $I$ is a {\it trace ideal} of $R$, that is $I = \tr_R(M)$ for some $R$-module $M$.
\item $I = \tr_R(I)$.
\item For each homomorphism $f : I \to R$ of $R$-modules, there is an endomorphism $g : I \to I$ such that $f = \iota{\cdot}g$, where $\iota : I \to R$ denotes the embedding.
\end{enumerate}
When $I$ is a regular ideal of $R$, one can add the following.
\begin{enumerate}[$(1)$]
\item[$(4)$] $I:I = R:I$. 
\end{enumerate}
\end{prop}

Examples of trace ideals are abundant; see e.g., \cite{DL, DMS, Dey, EGIM, GIK, HHS, HR, IK, KT, L, L2, LP}. If $R$ is a one-dimensional Cohen-Macaulay local ring such that $\rmQ(R)$ is Gorenstein, there is a one-to-one correspondence below:
\begin{eqnarray*}
\{\text{regular reflexive trace ideals}\}  &\longleftrightarrow& \{\text{reflexive birational module-finite extensions}\} \\
I &\mapsto& \End_R(I) = I:I \\
R:A &\mapsfrom& A.
\end{eqnarray*}
In addition, by \cite[Lemma 3.15]{GMP}, the maximal ideal $\m$ is a regular reflexive trace ideal, unless $R$ is a DVR. In one-dimensional Cohen-Macaulay local rings, there are two important classes of regular ideals, Ulrich ideals (\cite{GOTWY}) and good ideals (\cite{GIW}), and the relation between them can be described as follows; see \cite{D, DL, GOTWY}. 
$$
\xymatrix{
 \text{Ulrich ideals} \  \ar@{=>}[r] &\  \text{good ideals  $=$  regular stable trace ideals} \ar@{=>}[d]   \ \ar@{=>}[r] &   \ \text{regular trace ideals}
\\ 
  & \text{regular reflexive ideals}&    & 
}
$$
\begin{rem}
Regular trace ideals in {\it Arf rings}, defined by J. Lipman (\cite{Lipman}), are integrally closed (\cite[Theorem 7.4]{DL}, \cite[Proposition 3.1]{IK}), so they are stable, good, and reflexive. However, regular trace ideals are not necessarily reflexive in general; see e.g., \cite[Example 7.12]{DMS}, \cite[Example 3.11]{HR}. 
Moreover, if a one-dimensional Cohen-Macaulay ring does not have minimal multiplicity, the maximal ideal is not stable trace ideal even though it is a reflexive trace ideal. 
\end{rem}

Let $A$ be a birational extension of $R$, so $A$ is an intermediate ring between $R$ and $\rmQ(R)$. 
We now summarize some auxiliary results which we later need
throughout this paper. The following might be known, but let us include a brief proof for the sake of completeness (cf. {\cite[Proposition 4.14]{LW}}). 

\begin{lem}\label{2.2}
Let $X, Y$ be $A$-modules. Then the following assertions hold true. 
\begin{enumerate}[$(1)$]
\item  If $Y$ is torsionfree as an $R$-module, then $\Hom_R(X, Y) = \Hom_A(X, Y)$.   
\item The $A$-module $X$ is torsionfree if and only if $X$ is torsionfree as an $R$-module. 
\item Suppose that $X$ is torsionfree as an $R$-module and $X \ne (0)$. Then $X$ is indecomposable as an $A$-module if and only if $X$ is indecomposable as an $R$-module.
\end{enumerate}
\end{lem}

\begin{proof}
$(1)$ We only verify that every $R$-linear map $f: X \to Y$ is $A$-linear. Let $\alpha \in A$ and $x \in X$. We write $\alpha = \frac{a}{s}$ with $a, s \in R$ and $s$ is a non-zerodivisor on $R$. As $Y$ is $R$-torsionfree, the equalities
$$
s f(\alpha x) = f (s (\alpha x)) = f((s \alpha)x) = f(ax) = a f(x) = (s\alpha)f(x) = s (\alpha f(x)) 
$$
show that $f(\alpha x) = \alpha f(x)$. Hence $f$ is $A$-linear. 

$(2)$ Since $A$ is a birational extension of $R$, every non-zerodivisor on $A$ is, once $X$ is torsionfree as an $R$-module, a non-zerodivisor on $X$, and vice versa.

$(3)$ This follows from $(1)$. 
\end{proof}

\begin{lem}\label{2.3}
For an $A$-module $M$, we set $\fka = \tr_R(M)$. Then $\fka A = \fka$ and $\fka \subseteq R:A$. 
\end{lem}

\begin{proof}
Given every $R$-linear map $f: M \to R$, we have $f(M) \subseteq \fka$, so the composite map $g : M \overset{f}{\to} R \hookrightarrow A$ forms $A$-linear. Hence, for each $\alpha \in A$ and $x \in M$, we have
$$
\alpha  \tau(f \otimes x) = \alpha f(x) = \alpha g(x) = g(\alpha x) = f(\alpha x) \in f(M) \subseteq \fka. 
$$
This shows $\fka A \subseteq \fka$. Therefore $\fka A = \fka$ and $\fka \subseteq R:A$. 
\end{proof}

We close this section by proving the following (cf. \cite[Theorem 2.9]{DMS}). 

\begin{prop}\label{2.4}
Let $M$ be a reflexive $R$-module and $M \ne (0)$. We set $\fka = \tr_R(M)$.  Then $M$ is an $A$-module by extending the $R$-action if and only if $\fka A= \fka$, or equivalently, $\fka \subseteq R:A$.  
\end{prop}

\begin{proof}
Since $M$ is reflexive and $M \ne (0)$, we obtain $\fka  \ne (0)$ (\cite[Proposition 2.8 (vii)]{L}). If $M$ has an $A$-module structure by extending the $R$-action, then Lemma \ref{2.3} shows $\fka A = \fka$, so that $\fka \subseteq R:A$. On the other hand, suppose $\fka \subseteq R:A$. For every $f \in \Hom_R(M, R)$, the homomorphism
$g = \widehat{\alpha} \circ f : M \to \fka \to R$ is $R$-linear, where $\widehat{\alpha}: \fka \to R$ denotes the homothety map. So, $g \in \Hom_R(M, R)$ and $g(M) \subseteq \fka$. Therefore, as $g(x) = \alpha f(x)$ for all $x \in M$, we get $\fka A \subseteq \fka$, i.e., $\fka A = \fka$. We finally assume $\fka A = \fka$. The map $\psi : A \to \Hom_R(\fka, \fka)$, $\psi(\alpha) = \widehat{\alpha}$ is well-defined.  We now consider the $R$-algebra map
$$
A \overset{\psi}{\longrightarrow} \Hom_R(\fka, \fka) \overset{\varphi}{\longrightarrow} \Hom_R(M, M)
$$
where $\varphi$ is defined by the following way (see \cite[page 109]{L}): 
$$
\varphi : \Hom_R(\fka, \fka) \!=\! \Hom_R(\fka, R) \overset{\tau^*}{\hookrightarrow} \Hom_R(M^*\otimes_RM, R) \!\overset{\sim}{\rightarrow} \!\Hom_R(M, M^{**}) \!\overset{\sim}{\rightarrow} \!\Hom_R(M, M).
$$ 
We need the reflexivity of $M$ to construct the above morphism $\varphi$. 
Then $M$ is an $A$-module by the action $\alpha \underset{A}{\rightharpoonup} x = (\varphi(\widehat{\alpha}))x$ for all $\alpha \in A$ and $x \in M$. Moreover, we have $a \underset{A}{\rightharpoonup} x = ax$ for all $a \in R$ and $x \in M$. This completes the proof. 
\end{proof}


\section{Main results}\label{sec3}

First of all we fix the notation on which all the results in this section are based.

\begin{setting}\label{3.1}
Let $(R, \m)$ be a Cohen-Macaulay local ring with $\dim R=1$. Suppose that $\rmQ(R)$ is a Gorenstein ring. Let $I$ be a regular reflexive trace ideal of $R$. We set $A = \End_R(I)$ and identify $A$ with $I : I$. 
One can verify that $A = R:I \in \mod\, R$, $R \subseteq A \subseteq \overline{R}$, $R:A = I$, and $R:(R:A) = A$. For an ideal $J$ of $R$, we define $\OCM(R, J)$ the full subcategory of $\OCM(R)$ consisting of modules $M \in \OCM(R)$ satisfying $\tr_R(M)\subseteq J$. 
When $R$ admits the canonical module $\rmK_R$, an $R$-submodule $K$ of $\rmQ(R)$ is called a {\it fractional canonical ideal} of $R$, if $R \subseteq K \subseteq \overline{R}$ and $K \cong \rmK_R$ as an $R$-module. 
If $\rmK_R$ exists and the residue class field $R/\m$ is infinite, as is known by \cite[Proposition 2.3]{EGGHIKMT} and \cite[Corollary 2.8]{GMP}, the ring $R$ possesses a fractional canonical ideal if and only if $\rmQ(R)$ is Gorenstein; equivalently $\rmQ(\widehat{R})$ is a Gorenstein ring, where $\widehat{R}$ denotes the $\m$-adic completion of $R$; see \cite[Proposition 2.7]{GMP}. 
 We set $S = R[K]$ and $\fkc = R:S$. 
\end{setting}

In the following, simply saying that $R$ admits a fractional canonical ideal will assume the existence of the canonical module $\rmK_R$ without refusal. We denote by $\mu_R(-)$ the number of elements in a minimal system of generators.

With this notation the main result of this paper is stated as follows. 

\begin{thm}\label{main}
The following assertions hold true. 
\begin{enumerate}[$(1)$]
\item One has $\OCM(A) \subseteq \OCM(R, I) \subseteq \CM(A)$ as subcategories of $\mod\, R$.
\item The equality $\OCM(A) = \OCM(R, I)$ holds if and only if the ideal $I$ is stable. 
\item Suppose that $R$ admits a fractional canonical ideal $K$. Then the equality $\OCM(R, I) = \CM(A)$ holds if and only if $I \subseteq \fkc$, or equivalently, $IK = I$. 
\end{enumerate}
\end{thm}

\begin{proof}
$(1)$~Let $M \in \OCM(A)$. We may assume $M \ne (0)$ and may choose an exact sequence $0 \to M \to A^{\oplus n} \to A^{\oplus \ell}$ of $A$-modules with $n, \ell > 0$. Suppose that $R$ is a direct summand of $A^{\oplus n}$ as an $R$-module. By \cite[Proposition 2.8 (iii)]{L}, we can take an $R$-module $Y$ such that $A \cong R \oplus Y$. Since $A/R$ is torsion as an $R$-module, so is $Y$. This shows $Y=(0)$ because $Y \subseteq R \oplus Y \cong A$. Hence $A \cong R$ as an $R$-module. In particular, $\mu_R(A) = 1$. As $1 \not\in \m A$, we obtain $R=A$. Therefore $M \in \OCM(R)$ and $\tr_R(M)  \subseteq R=I$. 
Next, we consider the case where $R$ is not a direct summand of $A^{\oplus n}$. Because $A^{\oplus n}$ is $R$-reflexive, we get $A^{\oplus n} \in \OCM_{\calP}(R)$. We divide the exact sequence $0 \to M \to A^{\oplus n} \to A^{\oplus \ell}$ into the sequences
$$
0 \to M \to A^{\oplus n} \to X \to 0 \ \ \text{and} \ \ 0 \to X \to A^{\oplus \ell}
$$
of $A$-modules. Lemma \ref{2.2} (2) guarantees $X$ is a torsionfree $R$-module, that is $X \in \CM(R)$. By \cite[Lemma 2.6]{K}, we get $M \in \OCM_{\calP}(R) \subseteq \OCM(R)$. Let $\fka = \tr_R(M)$. Then 
$$
\fka \subseteq R:A = I
$$
by Proposition \ref{2.4}. Hence, in each case, we have $\OCM(A) \subseteq \OCM(R, I)$.


Let $M \in \OCM(R, I)$. May assume $M \ne (0)$. Since $\tr_R(M) \subseteq I = R:A$, the module $M$ has an $A$-module structure by Proposition \ref{2.4}. As $M$ is a torsionfree $R$-module, by Lemma \ref{2.2} (2), $M$ is torsionfree as an $A$-module. Thus $M \in \CM(A)$. Therefore we have
$\OCM(A) \subseteq \OCM(R, I) \subseteq \CM(A)$ as subcategories of $\mod\, R$. 

$(2)$~Suppose $\OCM(A) = \OCM(R, I)$. By \cite[Corollary 2.2]{GIK}, we note that $I = (R:I)I = \tr_R(I)$.  Because $I$ is $R$-reflexive, we get $I \in \OCM(R)$. So, $I \in \OCM(R, I) = \OCM(A)$. The ideal $I$ has an $A$-module structure and is reflexive as an $A$-module (remember that $\rmQ(A) = \rmQ(R)$ is a Gorenstein ring).  Since $A = I:I$, we obtain an isomorphism $I \cong A$ as an $A$-module (\cite[Lemma 2.9]{K}). We can choose $a \in I$ such that $I = a A$. Hence $I^2 = (aA)^2 = a (aA) = a I$. 
Conversely, we assume $I$ is stable, i.e.,  $I^2 = a I$ for some $a \in I$. Then 
$$
A = I:I = \frac{I}{a} \ \ \ \text{and} \ \ \ I = aA \cong A \ \ \ \text{in}\ \ \mod\, A.  
$$
For each $M \in \OCM(R, I)$, we have $\tr_R(M) \subseteq I = R:A$ and $M$ is a reflexive $R$-module.  By Proposition \ref{2.4}, $M$ has an $A$-module structure which naturally extends the structure of $R$-module. As $M \in \OCM(R)$, we may choose an exact sequence
$$
0 \to M \overset{\varphi}{\to} R^{\oplus n} \to R^{\oplus \ell}
$$
of $R$-modules, where $ n >0$ and $\ell \ge 0$. Then $\Im \varphi \subseteq I^{\oplus n}$ because $\tr_R(M) \subseteq I$, so we can consider an $R$-linear map
$$
\varphi' : M \to I^{\oplus n}
$$
defined by $\varphi'(x) = \varphi(x)$ for each $x \in M$. Since $I^{\oplus n}$ is an $A$-module and is torsionfree as an $R$-module, it is a torsionfree $A$-module. Hence the equality
$$
\Hom_R(M, I^{\oplus n}) = \Hom_A(M, I^{\oplus n})
$$ 
holds. Thus $\varphi' : M \to I^{\oplus n}$ is an $A$-linear map. 
Look at the commutative diagram
$$
\xymatrix{
0 \ \ar[r] &M \ar[r]^{\varphi} & R^{\oplus n}\ar[r] &X\ar[r] &  \ 0 \ \ \ \text{in} \ \ \mod\, R \\
0 \ \ar[r] &M \ar[r]^{\varphi'}\ar@{=}[u] & I^{\oplus n}\ar[r]\ar[u]^{i} &X'\ar[r] & \ 0 \ \ \ \text{in} \ \ \mod\, A
}
$$
of $R$-modules. We then have an injective $R$-linear map $h: X' \to X$, where $X' = \Coker \varphi' \in \mod\, A$. Hence $X'$ is a torsionfree $R$-module, so that it is torsionfree as an $A$-module. Consequently $M \in \OCM(A)$, as desired.

$(3)$~Suppose the existence of a fractional canonical ideal $K$ of $R$. Note that $A = I:I = R:I = (K:K):I = K:IK$ (\cite[Bemerkung 2.5]{HK}). Set $\rmK_A = K:A$. Then $\rmK_A = K:(K:IK) = IK$. We first assume $\OCM(R, I) = \CM(A)$. Then $IK=\rmK_A \in \OCM(R)$ and $\tr_R(IK) \subseteq I$. Choose an exact sequence
$$
0 \to IK \overset{\varphi}{\to} R^{\oplus n} \to R^{\oplus \ell}
$$
of $R$-modules, where $n>0$ and $\ell \ge 0$. We then have $\Im \varphi \subseteq I^{\oplus n}$ which induces the $R$-linear map $\varphi': IK \to I^{\oplus n}$ defined by $\varphi'(x) = \varphi(x)$ for each $x \in IK$. Because $I^{\oplus n}$ is $A$-torsionfree, the map $\varphi': IK \to I^{\oplus n}$ is $A$-linear. Look at the commutative diagram
$$
\xymatrix{
0 \ \ar[r] &IK \ar[r]^{\varphi} & R^{\oplus n}\ar[r] &X\ar[r] &  \ 0 \ \ \ \text{in} \ \ \mod\, R \\
0 \ \ar[r] &IK \ar[r]^{\varphi'} \ar@{=}[u] & I^{\oplus n}\ar[r]\ar[u]^{i} &X'\ar[r]\ar[u]^{h} & \ 0 \ \ \ \text{in} \ \ \mod\, A
}
$$
of $R$-modules, where $X' = \Coker \varphi' \in \mod\, A$. Then $\depth_RX'>0$, so we have an exact sequence
$$
0 \to \Hom_A(X', IK) \to \Hom_A(I^{\oplus n}, IK) \to \Hom_A(IK, IK) \to 0
$$
of $A$-modules, because $\Ext^1_A(X', IK) = (0)$. This shows the map $\varphi': IK \to I^{\oplus n}$ is a split monomorphism. Hence $IK$ is a direct summand of $I^{\oplus n}$ as an $A$-module. By choosing an $A$-module $X$ with $I^{\oplus n} \cong IK \oplus X$, we have
\begin{eqnarray*}
{\rm M}_n(A) &=& {\rm M}_n(I:I) \cong  {\rm M}_n(\Hom_A(I, I)) \cong \End_A(I^{\oplus n}) \cong \End_A(IK \oplus X) \\
&\cong&
\begin{bmatrix}
\Hom_A(IK, IK) & \Hom_A(X, IK) \\
\Hom_A(IK, X) & \Hom_A(X, X) \\
\end{bmatrix}
\end{eqnarray*}
where ${\rm M}_n(A)$ denotes the matrix ring of $n$-th square matrices whose entries are in $A$. Since ${\rm M}_n(A)$ is a finitely generated free $A$-module of rank $n^2$, the direct summand $Y=\Hom_A(X, IK)$ of ${\rm M}_n(A)$ is projective as an $A$-module. 
For each $\p \in \Ass A$, we have
$$
A_{\p} \oplus X_{\p} \cong (IK \oplus X)_{\p} \cong (I^{\oplus n})_{\p} \cong (I A_{\p})^{\oplus n} = A_{\p}^{\oplus n}
$$
because $A_{\p}$ is a Gorenstein ring. So $X_{\p}$ is a direct summand of $A_{\p}^{\oplus n}$, i.e., $X_{\p}$ is projective. Hence, $X_{\p}$ is a finitely generated free $A_{\p}$-module of rank $n-1$. This induces the isomorphisms
$$
Y_{\p} \cong \Hom_{A_{\p}}(X_{\p}, (IK)_{\p}) \cong \Hom_{A_{\p}}(X_{\p}, A_{\p}) \cong A_{\p}^{\oplus (n-1)}
$$
of $A_{\p}$-modules, which yield that $Y \cong A^{\oplus (n-1)}$ as an $A$-module because $Y$ is projective and $A$ is a semi-local ring. Therefore, we get the isomorphisms
$$
X \cong \Hom_A(\Hom_A(X, \rmK_A), \rmK_A) \cong \Hom_A(A^{\oplus (n-1)}, \rmK_A) \cong \rmK_A^{\oplus (n-1)}
$$
and hence $I^{\oplus n} \cong \rmK_A \oplus X \cong \rmK_A^{\oplus n}$ as an $A$-module. Since $IA_{P}$ is a faithful ideal possessing the injective dimension one, we have 
$$
IA_{P} \cong (\rmK_A)_P
$$
for every $P \in \Max A$. Therefore, we obtain the isomorphism
$$
\widehat{A} \otimes_A I \cong \widehat{A}\otimes_A \rmK_A
$$
of $\widehat{A}$-modules, where $\widehat{A}$ denotes the $\rmJ(A)$-adic completion of $A$. Here $\rmJ(A)$ is the Jacobson radical of $A$. Hence, as an $A$-module, we have $I \cong \rmK_A = IK$. This shows $IK = \alpha I$ for some unit $\alpha \in \rmQ(R)$. Thus, for all $\ell >0$, the equality $IK^{\ell} = \alpha^{\ell}I$ holds. As $S = R[K]$, by choosing $\ell \gg 0$ with $S = K^{\ell}$, we have $IS = \alpha^{\ell}I$, so that
$$
\alpha^{\ell} I =IS = IS \cdot S = \alpha^{\ell} I S.
$$
This implies $I = IS$, whence $I \subseteq I : S \subseteq R:S = \fkc$. 

Suppose $I \subseteq \fkc$. Notice that $\fkc = K:S$ and $R:\fkc = S$. Indeed, $\fkc = (K:K):S = K:KS = K:S$ and $R:\fkc = (K:K):\fkc = K:K\fkc= K:\fkc = K:(K:S) = S$. This shows $A = R:I \supseteq R:\fkc = S$. Since $I$ is a trace ideal of $R$ and $A = R:I$, we get $I \subseteq IK \subseteq IS \subseteq IA = I$. Hence $I = IK$. Conversely, if $IK = I$, then $IK^{\ell} = I$ for every $\ell > 0$. So, $IS = I$ and hence $I \subseteq R:S =\fkc$. Therefore, $I \subseteq \fkc$ if and only if $IK = I$. 

We finally assume $I \subseteq \fkc$ and prove $\OCM(R, I) = \CM(A)$. Indeed, let $M \in \CM(A)$ and consider $\fka = \tr_R(M)$. May assume $M \ne (0)$. By Lemma \ref{2.3}, we have $\fka \subseteq R:A = I$. Note that $M^{\vee} \in \CM(A)$, where $(-)^{\vee} = \Hom_A(-, \rmK_A)$.  Consider the exact sequence
$$
0 \to X \to A^{\oplus n} \to M^{\vee} \to 0
$$
of $A$-modules with $n > 0$. By applying the functor $(-)^{\vee}$ to the above sequence, we get
$$
0 \to M \to I^{\oplus n} \to X^{\vee} \to 0
$$
of $A$-modules because $M^{\vee\vee} \cong M$ and $(\rmK_A)^{\oplus n} = (IK)^{\oplus n} = I^{\oplus n}$. Since $I$ is reflexive as an $R$-module, we have $I^{\oplus n} \in \OCM(R)$. So we have an exact sequence
$$
0 \to I^{\oplus n} \to F_1 \to F_2
$$
of $R$-modules, where $F_1$ and $F_2$ are finitely generated free $R$-modules. By dividing the above sequence into the sequences
$0 \to I^{\oplus n} \overset{\alpha}{\to} F_1 \overset{\beta}{\to} Y \to 0$ and $0 \to Y \to F_2$, the pushout of $\alpha$ and $\beta$ gives a commutative diagram
$$
\xymatrix{
&  & 0 \ar[d] & 0 \ar[d] &  \\
0 \ar[r]& M \ar[r] \ar@{=}[d] & I^{\oplus n} \ar[r] \ar[d] & X^{\vee} \ar[r]\ar[d] & 0 \\
0 \ar[r]& M \ar[r] & F_1 \ar[r]\ar[d] & L \ar[r]\ar[d] & 0 \\
&  & Y \ar@{=}[r] \ar[d]& Y \ar[d] &  \\
&  & 0 & 0  &  
}
$$
with exact rows and columns. Since $X^{\vee}, Y \in \CM(R)$, we have $L \in \CM(R)$. Therefore $M \in \OCM(R)$ and hence $M \in \OCM(R, I)$. 
\end{proof}

We summarize some consequences. The next immediately follows from Theorem \ref{main}.

\begin{cor}[{cf. \cite[Theorem 7.9]{DMS}}]\label{3.3}
Suppose that $R$ admits a fractional canonical ideal $K$. Then the equality $\OCM(A) = \CM(A)$ holds if and only if  $I$ is stable and $I \subseteq \fkc$, or equivalently, $A$ is a Gorenstein ring. 
\end{cor}

\begin{rem}
By \cite[Proposition 2.8 (iii)]{L}, a finitely generated $R$-module $M$ does not have free summands if and only if $\tr_R(M)\subseteq \m$. Hence, $\OCM(R, \m)$ coincides with $\OCM_{\calP}(R)$.
\end{rem}

Theorem \ref{main} recovers the results of \cite[Theorem 5.1]{GMP} and \cite[Theorem 1.3]{K}.


\begin{cor}[{\cite[Theorem 5.1]{GMP}, \cite[Theorem 1.3]{K}}]\label{3.4}
Suppose that $R$ admits a fractional canonical ideal $K$ and that $R$ is not a DVR. Set $E=\m:\m$. Then the following assertions hold true. 
\begin{enumerate}[$(1)$]
\item The equality $\OCM(E) = \OCM_{\calP}(R)$ holds if and only if $\m$ is stable.
\item The equality $\OCM_{\calP}(R) = \CM(E)$ holds if and only if $R$ is an almost Gorenstein ring.
\item $E$ is a Gorenstein ring if and only if $R$ is an almost Gorenstein ring and $\m$ is stable. 
\end{enumerate}
\end{cor}

\begin{proof}
As $R$ is not a DVR, we have $E = \m:\m = R:\m$ (\cite[Lemma 3.15 (3)]{GMP}). Thus $\tr_R(\m) = \m$. Since $\OCM(R, \m) = \OCM_{\calP}(R)$, the assertions follow from Theorem \ref{main} and the fact that $R$ is almost Gorenstein if and only if $\m K = \m$ (\cite[Theorem 3.11]{GMP}). 
\end{proof}

As we show in the proof of Theorem \ref{main} (3), we have $\fkc = K:S$ and $R:\fkc = S$. This shows $S$ is reflexive as an $R$-module. Thus, $\fkc = R:S$ is a regular reflexive trace ideal. 

\begin{cor}\label{3.5}
Suppose that $R$ admits a fractional canonical ideal $K$. Then the following assertions hold true. 
\begin{enumerate}[$(1)$]
\item The equality $\OCM(R, \fkc) = \CM(S)$ holds.
\item $S$ is a Gorenstein ring if and only if the equality $\OCM(S) = \OCM(R, \fkc)$ holds.
\end{enumerate}
\end{cor}

\begin{proof}
As $(R:\fkc)\cdot\fkc = S\cdot\fkc = \fkc$, we get $\fkc = \tr_R(\fkc)$, i.e., $\fkc$ is a trace ideal of $R$. The ideal $\fkc$ is regular, and is reflexive as an $R$-module because $R:(R:\fkc) = R: S = \fkc$. Moreover we have $\fkc :\fkc = (R:S):\fkc = R:\fkc S = R:\fkc = S$. By Theorem \ref{main}, we conclude that $\OCM(R, \fkc) = \CM(S)$. This  proves the assertion $(1)$. The assertion $(2)$ follows from Corollary \ref{3.3}.
\end{proof}


\begin{rem}\label{3.7rem}
There is an alternative proof of Theorem \ref{main} using recent results in \cite{DMS, Dey}. Since we actually proved Theorem \ref{main} around five years ago, our proof does not use their results, but here we also record an alternative proof using their recent results.
Indeed, since $I=R:A$, $\rmQ(R)$ is Gorenstein, and $A \in \Ref(R)$, by \cite[Theorem 2.9]{DMS} (see also Proposition \ref{2.4}), we have $\OCM(R, I) = \OCM(R) \cap \CM(A)$; hence $\OCM(R, I) \subseteq \CM(A)$. Recently, Theorem \ref{main} (2) and $\OCM(A) \subseteq \OCM(R, I)$ have been independently proved by S. Dey using another approach; see \cite[Theorem 3.4, Proposition 2.5]{Dey}. In addition, the equality $\OCM(R, I) =\CM(A)$ holds if and only if $\CM(A) \subseteq \OCM(R)$. As $\OCM(R) = \Ref(R)$, by \cite[Theorem 5.5]{DMS}, the latter condition is equivalent to $K \subseteq A$, i.e., $S=R[K] \subseteq A$. This shows Theorem \ref{main} (3). 
\end{rem}

On the other hand, by using Theorem \ref{main}, we can provide an alternative proof of \cite[Theorem 1.1]{Dey} as well.


\begin{cor}[{\cite[Theorem 1.1]{Dey}}]
Let $A$ be a birational module-finite extension of $R$. Set $I = R:A$. Then the following conditions are equivalent.
\begin{enumerate}[$(1)$]
\item $A$ is a reflexive $R$-module and the ideal $I$ is stable.
\item $A \cong \Hom_R(A, R)$ as an $A$-module.
\item The equality $\Ref(A) = \CM(A) \cap \Ref(R)$ holds. 
\end{enumerate}
\end{cor}

\begin{proof}
Since $\rmQ(R)$ is Gorenstein, the ideal $I$ is reflexive as an $R$-module, while $I$ is regular because $A$ is a birational module-finite extension. The equalities $I = IA = (R:A) A = \tr_R(A)$ show that  $I$ is a regular reflexive trace ideal of $R$. Note that $\CM(A) \cap \Ref(R) = \OCM(R, I)$. In addition, if $A \in \Ref(R)$, then $A = I:I$. 

$(1) \Leftrightarrow (3)$ May assume $A \in \Ref(R)$. The equivalence follows from Theorem \ref{main} (2). 

$(1) \Rightarrow (2)$ We have $A=I:I$. The stability of $I$ implies $\Ref(A) = \OCM(R, I)$. So $I$ is reflexive as an $A$-module. By \cite[Lemma 2.9]{K}, we conclude that $A \cong I \cong \Hom_R(A, R)$. 

$(2) \Rightarrow (1)$ As $A$ is self-dual, it is reflexive. We choose $a \in I$ such that $I = a A$. Then $I^2 = (aA)^2 = a (aA) = a I$. Therefore, the ideal $I$ is stable. 
\end{proof}

We denote by $\calY_R$ the set of reflexive birational module-finite extensions of $R$ and $\calZ_R$ the set of regular reflexive trace ideals of $R$, respectively.
For each $M \in \mod\, R$, let $\left[M\right]$ stand for the isomorphism class of $M$, and for a subcategory $\calX$ of $\mod\, R$, $\ind\calX$ denotes the set of isomorphism classes of indecomposable objects in $\calX$.

\begin{cor}
Suppose that $R$ is a Gorenstein local domain with $\dim R=1$. Then the equalities
$$
\indOCM(R) = \bigcup_{R \ne A \in \calY_R}\indCM(A) \cup \left\{[R]\right\} 
=  \bigcup_{I \in \calZ_R, \ \! I \ne R} \!\! \indCM(\End_R(I)) \cup \left\{[R]\right\}
$$
hold.
\end{cor}

\begin{proof}
Let $M \in \OCM(R)$. We assume $M \ne (0)$ and it is indecomposable as an $R$-module. Set $I = \tr_R(M)$. As $M$ is reflexive, we then have $I \ne (0)$. If $I = R$, then $R$ is a direct summand of $M$ (\cite[Proposition 2.8 (iii)]{L}), so that $M \cong R$ as an $R$-module. 
We are now assuming that $M \not\cong R$. Then $I \ne R$. Since $R$ is Gorenstein, the ideal $I$ is reflexive as an $R$-module. So $I = R:(R:I)$. 
By taking $A = I:I$, the ring $A$ is a birational finite extension of $R$, $A \ne R$, and $M \in \OCM(R, I)$. By Theorem \ref{main}, we have $M \in \CM(A)$. In addition, by Lemma \ref{2.2} (3), $M$ is indecomposable as an $A$-module. Hence
$$
\indOCM(R) = \indCM(R) \subseteq \left[\bigcup_{R \ne A \in \calY_R}\indCM(A)\right] \cup \left\{[R]\right\}. 
$$
The converse follows from Lemma \ref{2.2}. The one-to-one correspondence between $\calZ_R$ and $\calY_R$ by sending $I \in \calZ_R$ to $\End_R(I) = I:I \in \calY_R$ shows the remaining equality.
\end{proof}




\section{When is $\OCM(R)$ of finite type?}\label{sec4}

The purpose of this section is to explore the question of when the category $\OCM(R)$ is of finite type, i.e., $R$ has only finitely many isomorphism classes of indecomposable reflexive $R$-modules. To state our results, we first recall the notion of generalized Gorenstein rings.

A generalized Gorenstein ring, introduced by S. Goto and S. Kumashiro \cite{GK}, is one of the attempts to generalize 
the class of almost Gorenstein rings. Similarly, for an almost Gorenstein ring, the notion is defined by a certain specific embedding of the ring $R$ into the canonical module $\rmK_R$ so that the difference $\rmK_R/R$ should be tame and well-behaved; see \cite{GK} for the precise definition. 
Since we are focusing on one-dimensional rings, let us recall the definition of generalized Gorenstein rings, in particular, of dimension one.

\begin{defn}[{\cite[Definition 1.2]{GK}}]\label{4.5}
Let $(R, \m)$ be a Cohen-Macaulay local ring with $\dim R=1$ admitting a fractional canonical ideal $K$ of $R$. 
We say that $R$ is a {\it generalized Gorenstein ring}, if either $R$ is Gorenstein, or $R$ is not a Gorenstein ring and $K/R$ is a free $R/\fka$-module, where $\fka = R:K$. 
\end{defn}

Remember that, by \cite[Theorem 3.11]{GMP}, $R$ is a non-Gorenstein almost Gorenstein ring if and only if $\fka = \m$. This indicates the generalized Gorenstein property is weaker than the almost Gorenstein property. We remark that $\rmQ(R)$ is Gorenstein, if either $R$ is a generalized Gorenstein ring (\cite[Lemma 3.1]{GTT}), or $R$ is reduced.

In this section, let $(R, \m)$ be a Cohen-Macaulay local ring with $\dim R=1$ admitting a fractional canonical ideal $K$. We set $S = R[K]$ and $\fkc = R:S$. For a Noetherian local ring $A$, let $\rme(A)$ be the Hilbert-Samuel multiplicity of $A$ and $v(A)$ the embedding dimension of $A$, respectively. 
We denote by $\ell_R(M)$ the length of an $R$-module $M$.  

Our answer to the above question can be stated as follows, where $|X|$ denotes the cardinality of a set $X$ and $X \ \dot\cup \ Y$ stands for the disjoint union of 
sets $X$ and $Y$.

\begin{thm}\label{4.1}
Suppose that $R$ is a generalized Gorenstein ring with minimal multiplicity. Then the equality 
$$
|\indOCM(R)| = \ell_R(R/\fkc) + |\indCM(S)|
$$
holds. Hence, $\OCM(R)$ is of finite type if and only if so is $\CM(S)$. 
\end{thm}

\begin{proof}
Note that $R$ is Gorenstein if and only if $R=S$ (\cite[Theorem 3.7]{GMP}), or equivalently, $\OCM(R) = \CM(R)$. Without loss of generality, we may assume $R$ is not a Gorenstein ring. Thus $v(R) = \rme(R) \ge 3$. 
For each integer $n \ge 0$, we define recursively
$$
R_n = 
\begin{cases}
\  R &  (n = 0) \\
\  R_{n-1}^{J(R_{n-1})}  & (n \geq 1)
\end{cases}
$$
where $J(R_{n-1})$ stands for the Jacobson radical of the ring $R_{n-1}$.
By \cite[Theorem 3.4]{CCGT}, the ring $S =R[K]$ is Gorenstein, $S=R_N$, and $R_n$ is a one-dimensional Cohen-Macaulay local ring with $v(R_n) = \rme(R_n) = \rme(R)$ for all $0 \le n < N$, where $N = \ell_R(R/\fkc)>0$. Since $R_1 = \m:\m$ and $\m$ is stable, we have $\OCM(R_1) = \OCM_{\calP}(R)$ by Corollary \ref{3.4} (1). Then
$$
\indOCM(R) = \indOCM(R_1) \ \dot\cup \ \{[R]\}.
$$
Indeed, let $M \in \OCM(R_1)$ such that $M$ is indecomposable as an $R_1$-module. Then $M \in \OCM_{\calP}(R)\subseteq \OCM(R)$ and $M$ is torsionfree as an $R$-module. So, $M$ is $R$-indecomposable by Lemma \ref{2.2} (3). Because $\OCM(R_1) = \OCM_{\calP}(R)$, we have $\indOCM(R_1) \cap \{[R]\} = \emptyset$. 
 On the other hand, let $M \in \OCM(R)$ such that $M$ is indecomposable as an $R$-module. This shows $M \cong R$ in $\mod\, R$ if $R$ is a direct summand of $M$. Otherwise, if $R$ is not a direct summand of $M$, we see that $M \in \OCM_{\calP}(R) = \OCM(R_1)$. Again, by Lemma \ref{2.2} (3), $M$ is indecomposable as an $R_1$-module. Hence, $\indOCM(R) = \indOCM(R_1) \ \dot\cup \ \{[R]\}$. Next, because $R_2 = J(R_1):J(R_1)$ and $J(R_1)$ is stable, we have $\OCM(R_2) = \OCM_{\calP}(R_1)$. Similarly as above, we obtain 
$$
\indOCM(R_1) = \indOCM(R_2) \ \dot\cup \ \{[R_1]\}.
$$
Repeating the same process for $R_n$ recursively induces the equality below
$$
\indOCM(R) = \indCM(S) \ \dot\cup \ \{[R], [R_1], \ldots, [R_{N-1}]\}
$$
because $S$ is a Gorenstein ring and $S=R_N$. Hence we have the equalities
$$
|\indOCM(R)| = N + |\indCM(S)| = \ell_R(R/\fkc) + |\indCM(S)|
$$
which complete the proof.   
\end{proof}

It is known by \cite{GK} that every Cohen-Macaulay local ring with multiplicity at most $3$ is a generalized Gorenstein ring. Hence we have the following.  

\begin{cor}\label{4.2}
Suppose that $\rme(R)=v(R)=3$. Then the equality 
$$
|\indOCM(R)| = \ell_R(R/\fkc) + |\indCM(S)|
$$
holds. 
\end{cor}

T. Kobayashi proved in \cite[Example 2.16]{K} that $\OCM(R)$ is of finite type for the semigroup ring $R=k[[t^3, t^7, t^8]]$ over a field $k$. We explore this example in more detail below. 

\begin{ex}
Let $V = k[[t]]$ be the formal power series ring over a field $k$ and set $R=k[[t^3, t^7, t^8]]$. Then $R$ is a generalized Gorenstein ring with minimal multiplicity. We set $K = R + Rt$ which is a fractional canonical ideal of $R$. Therefore $S=R[K] = V$ and $\fkc = R:V = t^{6}V$. Thus $\ell_R(R/\fkc) = \ell_R(V/\fkc) - \ell_R(V/R) = 6-4 = 2$. As $S=V$ is a DVR, we get $|\indCM(S)| = 1$. This shows $|\indOCM(R)| = 3$. Hence $\OCM(R)$ is of finite type. 
\end{ex}

Although we assume the ring has minimal multiplicity, the next corollary provides a formula regarding the cardinalities of $\indOCM(R)$ and $\indCM(\End_R(\m))$.  

\begin{cor}[{cf. \cite[Proposition 7.7]{DMS}, \cite[Corollary 1.4]{K}}]\label{4.3}
Suppose that $R$ is a non-Gorenstein almost Gorenstein ring with minimal multiplicity. Then the equality 
$$
|\indOCM(R)| = 1 +  |\indCM(\End_R(\m))| 
$$
holds. 
\end{cor}

\begin{proof}
Since $R$ is non-Gorenstein almost Gorenstein ring, we have $S = R[K] = \m:\m$ and $\fkc = \m$ (\cite[Lemma 3.5, Theorem 3.16]{GMP}). The assertion follows from Theorem \ref{4.1}. 
\end{proof}

We try to describe all the members of $\indOCM(R)$ for certain specific almost Gorenstein rings. Note that, if $\m \overline{R} \subseteq R$, then $R$ is an almost Gorenstein ring (\cite[Theorems 3.11]{GMP}). 
 We then have the following.

\begin{prop}\label{4.4}
Let $(R, \m)$ be a Cohen-Macaulay local ring with $\dim R=1$. Suppose that the normalization $\overline{R}$ is a DVR and is a module-finite extension of $R$. If $\m \overline{R} \subseteq R$, then the equality $\indOCM(R) = \left\{[R], [\overline{R}]\right\}$ holds. 
\end{prop}

\begin{proof}
We set $V=\overline{R}$. May assume $R$ is not a DVR, i.e., $R \ne V$. Hence $\m V= \m$ and $V = \m:\m$.  Note that $\rmQ(\widehat{R})$ is a Gorenstein ring, where $\widehat{R}$ denotes the $\m$-adic completion of $R$. Thus, by \cite[Satz 6.21]{HK} (see also \cite[Proposition 2.7]{GMP}), $R$ contains a canonical ideal $I$, i.e., $I ~(\ne R)$ is an ideal of $R$ and $I \cong \rmK_R$ as an $R$-module.  By choosing $a \in I$ with $IV=aV$, the ideal $(a)$ is a reduction of $I$. Consider 
$$
K = \frac{I}{a} = \left\{\frac{x}{a} ~\middle|~ x \in I\right\} \subseteq \rmQ(A)
$$
which is a fractional canonical ideal of $R$. We choose $b \in \m$ satisfying $\m V = b V$. Then $\m= bV$, so that $\m^2 = (b V)^2 = b(bV) = b\m$. By Corollary \ref{3.4}, we obtain $\OCM(V) = \OCM_{\calP}(R) = \CM(V)$. This implies $\indOCM(R) = \left\{[R], [V]\right\}$ as desired. 
\end{proof}

Let us note some examples in order to illustrate Proposition \ref{4.4}.

\begin{ex}
Let $A$ be a regular local ring with $n = \dim A \ge 2$. Let $X_1,X_2, \ldots, X_n$ be a regular system of parameters of $A$ and set $P_i =(X_j \mid 1 \le j \le n, ~j \ne i)$ for each $1 \le i \le n$. We set
$R = A/ \bigcap_{i=1}^n P_i$. Then the equality $\indOCM(R) = \left\{[R], [\overline{R}]\right\}$ holds. 
\end{ex}

\begin{proof}
Let $x_i$ denote the image of $X_i$ in $R$. We consider $\fkp_i =(x_j \mid 1 \le j \le n,~j \ne i)$ and $V = \prod_{i=1}^n(R/\fkp_i)$. The homomorphism
$\varphi : R \to V, ~a \mapsto (\overline{a}, \overline{a}, \ldots, \overline{a})$
is injective, $V = \overline{R}$, and $\m V = \m$. Hence $\indOCM(R) = \left\{[R], [V]\right\}$. 
\end{proof}

For a commutative ring $R$ containing a field of positive characteristic $p >0$,  we denote $R$ by $T$ when we regard $R$ as an $R$-algebra via the Frobenius map $F : R \to R,\  a \mapsto a^p.$ Note that $R$ is a reduced ring if and only if $F : R \to T$ is injective. We say that the ring $R$ is {\it $F$-pure} if for each $R$-module $M$ the homomorphism $M \to T \otimes_R M, \  m \mapsto 1 \otimes m$ is injective.

\begin{ex}
Let $(R, \m)$ be a Cohen-Macaulay local ring with $\dim R=1$. Suppose that $R$ is $F$-pure and $\overline{R}$ is a DVR. Then the equality $\indOCM(R) = \left\{[R], [\overline{R}]\right\}$ holds. 
\end{ex}

\begin{proof}
Since the $\m$-adic completion $\widehat{R}$ of $R$ remains $F$-purity (\cite[Lemma 3.26]{Hashimoto}), $\widehat{R}$ is a reduced ring, so that the normalization $\overline{R}$ is a module-finite extension of $R$, whence $\ell_R(\overline{R}/R) < \infty$, that is $\m^{\ell} {\cdot}(\overline{R}/R) =(0)$ for some $\ell \gg 0$. 
The Frobenius map naturally induces the homomorphism $f :\operatorname{Q}(R)/R \to \operatorname{Q}(R)/R, \ \overline{a} \mapsto \overline{a^p}$ of $R$-modules, where $\overline{a}$ (resp. $\overline{a^p}$) denotes, for each $a \in \operatorname{Q}(R)$, the image of $a$ in $\operatorname{Q}(R)/R$ (resp. the image of $a^p$ in $\operatorname{Q}(R)/R$). The homomorphism $f :\operatorname{Q}(R)/R \to \operatorname{Q}(R)/R$ is injective, once $R$ is an $F$-pure ring.
Hence, $f^{\ell}(\m {\cdot}(\overline{R}/R)) =(0)$ for all $\ell \gg 0$. 
Therefore, the injectivity of the homomorphism $f :\operatorname{Q}(R)/R \to \operatorname{Q}(R)/R$ guarantees that $\m {\cdot}(\overline{R}/R)=(0)$. 
\end{proof}

Computing the ring $S=R[K]$ and the length $\ell_R(R/\fkc)$ is not so difficult, especially for numerical semigroup rings. We summarize some terminology.  

Let $0 < a_1, a_2, \ldots, a_\ell \in \Bbb Z~(\ell >0)$ be integers such that $\mathrm{GCD}~(a_1, a_2, \ldots, a_\ell)=1$. Set 
$$
H = \left<a_1, a_2, \ldots, a_\ell\right>=\left\{\sum_{i=1}^\ell c_ia_i \mid 0 \le c_i \in \Bbb Z~\text{for~all}~1 \le i \le \ell \right\}
$$
and call it {\it the numerical semigroup generated by the numbers $\{a_i\}_{1 \le i \le \ell}$}. Let $V = k[[t]]$ be the formal power series ring over a field $k$. We set 
$$
R = k[[H]] = k[[t^{a_1}, t^{a_2}, \ldots, t^{a_\ell}]]
$$ 
in $V$ and call it {\it the semigroup ring of $H$ over $k$}. The ring  $R$ is a one-dimensional Cohen-Macaulay local domain with $\overline{R} = V$ and $\m = (t^{a_1},t^{a_2}, \ldots, t^{a_\ell} )$, the maximal ideal. 
Let 
$$
\rmc(H) = \min \{n \in \Bbb Z \mid m \in H~\text{for~all}~m \in \Bbb Z~\text{such~that~}m \ge n\}
$$
be {\it the conductor of $H$} and set $\rmf(H) = \rmc(H) -1$. Hence, $\rmf(H) = \max ~(\Bbb Z \setminus H)$, which is called {\it the Frobenius number of $H$}. Let $$\mathrm{PF}(H) = \{n \in \Bbb Z \setminus H \mid n + a_i \in H~\text{for~all}~1 \le  i \le \ell\}$$ denote the set of pseudo-Frobenius numbers of $H$. Therefore, $\rmf(H)$ coincides with the $\rma$-invariant of the graded $k$-algebra $k[t^{a_1}, t^{a_2}, \ldots, t^{a_\ell}]$ and $| \mathrm{PF}(H)| = \rmr(R)$ (\cite[Example (2.1.9), Definition (3.1.4)]{GW}), where $\rmr(R)$ denotes the Cohen-Macaulay type of $R$.  We set  $f = \rmf(H)$ and 
$$
K = \sum_{c \in \mathrm{PF}(H)}Rt^{f-c}
$$
in $V$. Then $K$ is a fractional ideal of $R$ such that $R \subseteq K \subseteq \overline{R}$ and 
$$
K \cong \rmK_R = \sum_{c \in \mathrm{PF}(H)}Rt^{-c}
$$
as an $R$-module (\cite[Example (2.1.9)]{GW}). Hence $K$ is a fractional canonical ideal of $R$. 
The example below satisfies $\m V \subseteq R$. 

\begin{ex}
Let $R = k[[t^{a},t^{a+1},  \ldots, t^{2a-1}]]$~$(a \ge 2)$. Then $\indOCM(R) = \left\{[R], [\overline{R}]\right\}$. 
\end{ex}

Recall that a Cohen-Macaulay local ring $A$ is said to be {\it of finite CM-representation type}, if $\CM(A)$ is of finite type, i.e., there are only a finite number of isomorphism classes of indecomposable MCM $A$-modules. Remember that $\mu_R(-)$ is the number of elements in a minimal system of generators. 

With the notation of above, we have the following. 


\begin{lem}\label{4.8}
Let $R = k[[H]]$ be the semigroup ring of a numerical semigroup $H$ over a field $k$. Then the following conditions are equivalent.
\begin{enumerate}[$(1)$]
\item $R$ is a Gorenstein ring of finite CM-representation type. 
\item $H$ has one of the following forms:
\begin{enumerate}[$(a)$]
\item $H=\Bbb N$,
\item $H = \left<2, 2q + 1\right>~(q \ge 1)$,
\item $H=\left<3, 4\right>$, or
\item $H=\left<3, 5\right>$.
\end{enumerate}
\end{enumerate}
\end{lem}

\begin{proof}
By \cite[Theorem 4.10]{LW}, the ring $R$ is of finite CM-representation type if and only if $\mu_R(V) \le 3$ and $\mu_R([\m V + R]/R) \le 1$. Set $e = \rme(R)$. Then $e = \min\{a_i \mid 1 \le i \le \ell\}$ and $\m V = t^e V$. So $\mu_R(V) = \ell_R(V/\m V) = \ell_R(V/t^eV) = e$. Hence, without loss of generality, we may assume $\mu_R(V) \le 3$. If $\mu_R(V)=1$, then $H = \Bbb N$ and $R=V$. If $\mu_R(V)=2$, then $H=\left<2, 2q + 1\right>$ with $q \ge 1$. As $H$ is symmetric, the ring $R$ is Gorenstein. Note that $V=R + Rt$. So $\m V + R = t^2V + R = R + Rt^3$. This implies $\mu_R([\m V + R]/R) = 1$. Thus $R$ is of finite CM-representation type. We consider the case where $\mu_R(V) = 3$. Suppose that $R$ is a Gorenstein ring of finite CM-representation type. Then $\rme(R) \ne v(R)$, so that $v(R) = 2$. Hence, either $H=\left<3, 4\right>$, or $\left<3, 5\right>$. Conversely, we are assuming that $H=\left<3, 4\right>$ (resp. $H=\left<3, 5\right>$). Then $R$ is a Gorenstein ring because $H$ is symmetric. As $V = R + Rt + Rt^2$, we see that $\m V + R = R + Rt^5$ (resp. $\m V + R = R + Rt^4$). Therefore $\mu_R([\m V + R]/R) = 1$, and hence $R$ is of finite CM-representation type. 
\end{proof}

\begin{rem}
Let $R=k[[H]]$ be the semigroup ring of a numerical semigroup $H$ over a field $k$. Assume that $R$ is Gorenstein and $k$ is infinite. By \cite[Corollary 3.16]{HR} and Lemma \ref{4.8}, the ring $R$ is of finite CM-representation type if and only if 
 $R$ is {\it small}, i.e., it has only finitely many trace ideals, or equivalently, all trace ideals are integrally closed. 
\end{rem}

We apply Theorem \ref{4.1} and Lemma \ref{4.8} to get the following. 

\begin{cor}\label{4.9}
Let $R$ be the numerical semigroup ring over a field $k$.
Suppose that $R$ is a generalized Gorenstein ring with minimal multiplicity. 
Then the following conditions are equivalent.
\begin{enumerate}[$(1)$]
\item The category $\OCM(R)$ is of finite type.
\item $S = k[[H]]$ is a semigroup ring of $H$, where $H$ is one of the following forms:
\begin{enumerate}[$(a)$]
\item $H=\Bbb N$,
\item $H = \left<2, 2q + 1\right>~(q \ge 1)$,
\item $H=\left<3, 4\right>$, or
\item $H=\left<3, 5\right>$.
\end{enumerate}
\end{enumerate}
\end{cor}

\begin{proof}
The ring $S=R[K]$ is Gorenstein (see the proof of Theorem \ref{4.1}) and is the numerical semigroup ring over $k$. By Theorem \ref{4.1}, the set $\indOCM(R)$ is finite if and only if so is $\indCM(S)$, i.e., $S$ is of finite CM-representation type. The latter condition is equivalent to saying that the corresponding semigroup of $S$ has one of the forms stated in the assertion $(2)$.  
\end{proof}


\section{Relation between Ulrich ideals}\label{sec5}

This section aims at exploring the relation between the finiteness of $\indOCM(R)$ and the set of Ulrich ideals. 
The notion of Ulrich ideals is one of the modifications of that of stable maximal ideal introduced in 1971 by his monumental paper \cite{Lipman} of J. Lipman. The present modification was formulated by S. Goto, K. Ozeki, R. Takahashi, K.-i. Watanabe, and K.-i. Yoshida \cite{GOTWY} in 2014, where the authors developed the basic theory, revealing that the behavior of Ulrich ideals has some ample information about the singularities of base rings. If a Cohen-Macaulay local ring $A$ has finite CM-representation type, then $A$ contains only finitely many Ulrich ideals (\cite[Theorem 7.8]{GOTWY}). In a one-dimensional non-Gorenstein almost Gorenstein local ring, the only possible Ulrich ideal is the maximal ideal (\cite[Theorem 2.14]{GTT2}). In \cite{GIT} the authors explored the ubiquity of Ulrich ideals in a {\it $2$-AGL} rings (one of the generalizations of Gorenstein rings of dimension one) and showed that the existence of two-generated Ulrich ideals reflects a rather strong restriction on the structure of base rings (\cite[Theorem 4.7]{GIT}). Moreover, over a Gorenstein local ring $(A, \m)$, if $I$ and $J$ are Ulrich ideals of $A$ with $\m J \subseteq I \subsetneq J$, then $A$ must be a hypersurface (\cite[Corollary 7.5]{GOTWY}).


Let $(R, \m)$ be a Cohen-Macaulay local ring with $\dim R=1$. We assume that $\rmQ(R)$ is a Gorenstein ring.  
An $\m$-primary ideal $I$ of $R$ is called {\it Ulrich ideal} if $I^2 = aI$ and $I/(a)$ is free as an $R/I$-module for some $a \in I$ (\cite[Definition 1.1]{GOTWY}).
We denote by $\calX_R$ the set of Ulrich ideals of $R$. Recall that $\rmr(R)$ denotes the Cohen-Macaulay type of $R$.

The following generalizes \cite[Theorem 7.8]{GOTWY} for one-dimensional Cohen-Macaulay rings. 

\begin{thm}\label{4.10}
If the category $\OCM(R)$ is of finite type, then $\calX_R$ is a finite set. 
\end{thm}

\begin{proof}
We may assume $\calX_R \ne \emptyset$. We define 
$$
\calW_R = \left\{\left[\Syz_2^R(R/I)\right] \mid I \in \calX_R\right\}
$$ 
where $\left[\Syz_2^R(R/I)\right]$ denotes the isomorphism class of the second syzygy module $\Syz_2^R(R/I)$ of $R/I$. By \cite[Corollary 7.7]{GOTWY}, there is a one-to-one correspondence below
$$
\calX_R \to \calW_R, \ I \mapsto \left[\Syz_2^R(R/I)\right]. 
$$
Let $I \in \calX_R$ and set $M = \Syz_2^R(R/I) \in \OCM(R)$. Then $M \ne (0)$ and $\mu_R(M) = t(t+1)$, where $t = \mu_R(I)-1 >0$ (\cite[Theorem 7.1]{GOTWY}). By \cite[Theorem 2.5]{GTT2}, we get $t\cdot \rmr(R/I) =\rmr(R)$, whence $t \le \rmr(R)$.  Therefore
$$
\mu_R(M) = t(t+1) \le \rmr(R)\cdot(\rmr(R) + 1)
$$
so the minimal number $\mu_R(M)$ of generators of $M$ has an upper bound which is independent of the choice of $I \in \calX_R$. Hence the set $\calW_R$ is finite because $\indOCM(R)$ is finite. This yields $\calX_R$ is a finite set.
\end{proof}

We explore some examples. 

\begin{ex}\label{4.11}
Let $(A, \m)$ be a Cohen-Macaulay local ring with $\dim A=1$ admitting the canonical module $\rmK_A$. Suppose that $\rmQ(A)$ is a Gorenstein ring and $A$ has an infinite residue class field. We set $R = A \ltimes A$ the idealization of $A$ over $A$. Then, because $\calX_R$ is infinite (\cite[Example 2.2]{GOTWY}), we have $|\indOCM(R)| = \infty$.
\end{ex}

As we show next, the converse of Theorem \ref{4.10} does not hold in general. 

\begin{ex}\label{4.12}
Let $R =k[[t^3, t^7]]$ be the semigroup ring over a field $k$. Then, by \cite[Example 2.7]{EG}, we have
$$
\calX_R = \{(t^6-ct^7, t^{10}) \mid 0 \ne c \in k\}. 
$$
Hence $\calX_R$ is a finite set if $k$ is finite. However, by Lemma \ref{4.8}, we get $|\indOCM(R)| = |\indCM(R)| = \infty$. 
\end{ex}


\section{Relation between analytically unramifiedness}\label{sec6}

In this section, let $(R, \m)$ be a Cohen-Macaulay local ring with $\dim R=1$. 
Let $A$ be a birational finite extension of $R$. Set $I = R:A$. 
Since $A$ is a fractional ideal of $R$, the ideal $I$ is regular. 
Note that $I = R:A \cong \Hom_R(A, R)$ is reflexive. Thus $I = R:(R:I)$. Moreover, because $I = IA = (R:A) A = \tr_R(A)$, we see that $I$ is a trace ideal of $R$. Hence, $I$ is a regular reflexive trace ideal of $R$.

We explore a birational module-finite extension $\widetilde{A} = I:I$. The ring coincides with $A$ if it is reflexive. Otherwise, in general, we have $R \subseteq A \subseteq \widetilde{A}  \subseteq \overline{R}$. 
Hence
$\widetilde{A} =R:I$, $R:\widetilde{A} = I$, and $\widetilde{A}  = R:(R: \widetilde{A})$. So, the ring $\widetilde{A}$ is reflexive as an $R$-module. 


\begin{lem}\label{5.1}
Suppose that the normalization $\overline{R}$ is not finitely generated as an $R$-module. Then there exists a chain of intermediate rings 
$$
R \subsetneq A_1 \subsetneq A_2 \subsetneq \ldots \subsetneq A_{\ell} \subsetneq \ldots \subsetneq \overline{R} 
$$
between $R$ and $\overline{R}$ such that $A_i$ is a finitely generated reflexive $R$-module for all $i \ge 1$. 
\end{lem}

\begin{proof}
As $R \ne \overline{R}$, we choose a birational module-finite extension $A$ of $R$ such that $R \subsetneq A \subseteq \overline{A}$. We set $I = R: A$ and consider $A_1 = I:I$. Then $R \subsetneq A \subseteq A_1 \subseteq \overline{R}$ and $A_1$ is a finitely generated reflexive $R$-module. Because $\overline{A_1} = \overline{R}$ and $A_1$ is not regular, repeating the same process for $A_1$ recursively gives an infinite chain of birational module-finite extensions
$$
R \subsetneq A_1 \subsetneq A_2 \subsetneq \ldots \subsetneq A_{\ell} \subsetneq \ldots \subsetneq \overline{R} 
$$
where $A_i$ is reflexive as an $R$-module for all $i \ge 1$. 
\end{proof}

\begin{lem}\label{5.2}
Let $A$ be a birational module-finite extension of $R$. Then, for each non-zero finitely generated $A$-module $X$, $X$ has an indecomposable decomposition as an $A$-module. 
\end{lem}

\begin{proof}
Suppose the contrary and choose a counterexample $X \in \mod\, A$ so that $\mu_R(X)$ is as small as possible. Let $Y$ and $Z$ be non-zero $A$-modules with $X \cong Y \oplus Z$ as an $A$-module. Then $\mu_R(Y) + \mu_R(Z) = \mu_R(X)$ and $\mu_R(Y)$, $\mu_R(Z)$ are positive. The minimality of $\mu_R(X)$ guarantees that the $A$-modules $Y$ and $Z$ possess indecomposable decompositions as $A$-modules, so does $X$. This is absurd. 
\end{proof}

Recall that a Noetherian local ring $(A, \m)$ is called analytically unramified, if the $\m$-adic completion $\widehat{A}$ is reduced. Note that, for a one-dimensional Cohen-Macaulay local ring $A$, $A$ is analytically unramified if and only if the normalization $\overline{A}$ is finite over $A$; see e.g., \cite[Theorem 4.6]{LW}.  
The following theorem provides a generalization of the fact that every one-dimensional Cohen-Macaulay local ring of finite CM-representation type is analytically unramified; see e.g., \cite[Proposition 4.15]{LW}. 

\begin{thm}[{cf. \cite[Lemma 6.6]{DMS}}]\label{5.4}
If the category $\Ref(R)$ is of finite type, then $R$ is analytically unramified. 
\end{thm}

\begin{proof}
Suppose the normalization $\overline{R}$ is not finitely generated. 
Choose a chain of birational module-finite extensions of $R$ 
$$
R \subsetneq A_1 \subsetneq A_2 \subsetneq \ldots \subsetneq A_{\ell} \subsetneq \ldots \subsetneq \overline{R} 
$$
such that $A_i$ is reflexive as an $R$-module for all $i \ge 1$. For each $i \ge 1$, we denote by
$$
A_i= \bigoplus_{\alpha=1}^{n_1}M_{i, \alpha}
$$
the indecomposable decomposition of $A_i$, where $n_i > 0$ and $M_{i, \alpha}~ (\ne (0))$ is an ideal of $A_i$ that is indecomposable as an $A_i$-module. For each integers $i < j$, $1 \le \alpha \le n_i$, and $1 \le \beta \le n_j$, if $M_{i, \alpha} \cong M_{j, \beta}$ as an $R$-module, then $A_j {\cdot} M_{i, \alpha} \subseteq M_{i, \alpha}$ in $\rmQ(R)$. 

We fix an integer $i \ge 1$. Suppose that, for each $1 \le \alpha \le n_i$, there exist $j > i$ and $1 \le \beta \le n_j$ satisfying $M_{i, \alpha} \cong M_{j, \beta}$ as an $R$-module. Then $A_j {\cdot} M_{i, \alpha} \subseteq M_{i, \alpha}$ in $\rmQ(R)$. As $j \ge i+1$, we obtain
$$
A_{i+1}{\cdot}M_{i, \alpha} \subseteq A_j{\cdot}M_{i, \alpha} \subseteq M_{i, \alpha} \subseteq A_i.
$$
This yields $A_{i+1}{\cdot}A_i \subseteq A_i$, because $A_i= \bigoplus_{\alpha=1}^{n_1}M_{i, \alpha}$. Hence $A_{i+1} = A_i$, which makes a contradiction. Therefore, there exists $1 \le \alpha \le n_i$ such that $M_{i, \alpha} \not\cong M_{j, \beta}$ as an $R$-module for all $j > i$ and $1 \le \beta \le n_j$. Thus the set of isomorphism classes of $M_{i, \alpha}$ below
$$
\left\{\left[ M_{i, \alpha} \right] \mid i \ge 1, \ 1 \le \alpha \le n_i \right\}
$$
is infinite. Since $M_{i, \alpha}$ is a reflexive $R$-module, it is $R$-indecomposable by Lemma \ref{2.2} (3). This is impossible, because the set $\ind \Ref(R)$ is finite. 
\end{proof}

If the category $\Ref(R)$ is of finite type, then Theorem \ref{5.4} shows $R$ is reduced; equivalently $R$ is {\it an isolated singularity}, i.e., the local ring $R_{\p}$ is regular for every non-maximal prime ideal $\p$ of $R$. 

Closing this paper we consider the converse of Theorem \ref{5.4}. To do this, 
Arf rings, defined by J. Lipman \cite{Lipman} in 1971, play a key role in our argument.  
An Arf ring has been properly generalized for a class of rings, which is used for classification of curve singularities, studied by C. Arf \cite{Arf} in 1949. 
For a Noetherian semi-local ring $R$ such that $R_M$ is a one-dimensional Cohen-Macaulay local ring for every $M \in \Max R$, we say that $R$ is {\it an Arf ring} if the following conditions are satisfied:
\begin{enumerate}[$(1)$]
\item Every integrally closed regular ideal $I$ in $R$ has a principal reduction, i.e., $I^{n+1} = a I^n$ for some $n \ge 0$ and $a \in I$.
\item If $x, y, z \in R$ such that $x$ is a non-zerodivisor on $R$ and $y/x, z/x \in \overline{R}$, then $yz/x \in R$. 
\end{enumerate}

Despite there are various assumptions on base rings, the following provides alternative proof of the results: \cite[Corollary 4.5]{D}, \cite[Corollary 7.9]{DL}, and \cite[Corollary 3.6]{IK}. 



\begin{cor}[{\cite[Corollary 3.5]{D}, \cite[Corollary 7.9]{DL}, \cite[Corollary 3.6]{IK}}]\label{6.4}
Suppose that $\overline{R}$ is a local ring and $\rmQ(R)$ is Gorensein. If $R$ is an analytically unramified Arf ring, then the category $\OCM(R)$ is of finite type. 
\end{cor}


\begin{proof}
Let $\ell = \ell_R(\overline{R}/R) < \infty$. 
For each integer $n \ge 0$, we consider
$$
R_n = 
\begin{cases}
\  R &  (n = 0) \\
\  R_{n-1}^{J(R_{n-1})}  & (n \geq 1)
\end{cases}
$$
where $J(R_{n-1})$ stands for the Jacobson radical of the ring $R_{n-1}$. Since $\overline{R}$ is a local ring, so is the ring $R_n$ for all $n \ge 0$. Hence, by \cite[Theorem 2.2]{Lipman}, $v(R_n) = \e(R_n)$ for all $n \ge 0$. In particular, $R_1 = \m :\m$. If $\ell = 0$, then $R$ is a DVR. This shows the set $\indOCM(R) = \indCM(R) =\{[R]\}$ is finite. We assume $\ell > 0$ and the assertion holds for $\ell - 1$. Note that $R \ne R_1$ (\cite[Lemma 3.15]{GMP}). Since $R_1$ is an Arf ring and $\ell_{R_1}(\overline{R}/R_1) < \infty$, the hypothesis of induction on $\ell$ shows $\indOCM(R_1)$ is a finite set. Hence, because $\OCM(R_1) = \OCM_{\calP}(R)$ (\cite[Theorem 1.3 (2)]{K}), we conclude that $\indOCM(R) = \{[R]\} \ \dot\cup \ \indOCM(R_1)$. Consequently, $\indOCM(R)$ is finite. 
\end{proof}


\begin{rem}
Even though $R$ is analytically unramified, $\OCM(R)$ is not necessarily of finite type unless $R$ is an Arf ring. Indeed, let $R = k[[t^3, t^7]]$ be the semigroup ring over a field $k$. As $R$ does not have minimal multiplicity, the ring is not  Arf. However, by Lemma \ref{4.8}, we have $|\indOCM(R)|= |\indCM(R)|=\infty$. In addition, the converse of Corollary \ref{6.4} does not hold in general. The ring $R=k[[t^3, t^4]]$ is not an Arf ring, but $|\indOCM(R)| = |\indCM(R)|<\infty$. 
\end{rem}



As a consequence of Corollary \ref{6.4}, there are only finitely many Ulrich ideals in Arf local rings once the normalization is a local ring. 

\begin{sdc} 
Competing Interests:  The authors have no competing interests to declare that are relevant to the content of this article.
\end{sdc}

\begin{ac}
The authors are grateful to Souvik Dey for suggesting {\it Remark} \ref{3.7rem} as well as his comments on a first draft of this paper.
\end{ac}


\begin{thebibliography}{20}

\bibitem{Bass}
{\sc H. Bass}, Finitistic dimension and a homological generalization of semi-primary rings, {\em Trans. Amer. Math. Soc.}, {\bf 95} (1960), 466--488.

\bibitem{Arf}
{\sc C. Arf}, Une interpr\'{e}tation alg\'{e}brique de la suite des ordres de multiplicit\'{e} d'une branche alg\'{e}brique, {\em Proc. London Math. Soc.}, Series 2, {\bf 50} (1949), 256--287.

\bibitem{BF}
{\sc V. Barucci and R. Fr\"{o}berg}, One-dimensional almost Gorenstein rings, {\em J. Algebra}, {\bf 188} (1997), no. 2, 418--442.


\bibitem{CCGT}
{\sc E. Celikbas, O. Celikbas, S. Goto, and N. Taniguchi}, Generalized Gorenstein Arf rings, {\em Ark. Mat.}, {\bf 57} (2019), no.1, 35--53.

\bibitem{CGKM}
{\sc T. D. M. Chau, S. Goto, S. Kumashiro, and N. Matsuoka}, Sally modules of canonical ideals in dimension one and $2$-AGL rings, {\em J. Algebra}, {\bf 521} (2019), 299--330.


\bibitem{D}
{\sc H. Dao}, Reflexive modules, self-dual modules and Arf rings, arXiv:2105.12240. 

\bibitem{DITV}
{\sc H. Dao, O. Iyama, R. Takahashi, and C. Vial}, Non-commutative resolutions and Grothendieck groups, {\em J. Noncommut. Geom.}, {\bf 9} (2015), no. 1, 21--34.

\bibitem{DL}
{\sc H. Dao and H. Lindo}, Stable trace ideals and applications, arXiv:2106.07064v2. 

\bibitem{DMS}
{\sc H. Dao, S. Maitra, and P. Sridhar}, On reflexive and $I$-Ulrich modules over curve singularities, {\em Trans. Amer. Math. Soc.} (to appear). 

\bibitem{Dey}
{\sc S. Dey}, Finite birational extension with stable conductor, arXiv:2212.09087. 

\bibitem{EGGHIKMT}
{\sc N. Endo, L. Ghezzi, S. Goto, J. Hong, S.-i. Iai, T. Kobayashi, N. Matsuoka, and R. Takahashi}, Rings with $q$-torsionfree canonical modules, arXiv:2301.02635.

\bibitem{EG}
{\sc N. Endo and S. Goto}, Ulrich ideals in numerical semigroup rings of small multiplicity, {\em J. Algebra}, {\bf 611} (2022), 435--479. 

\bibitem{EGIM}
{\sc N. Endo, S. Goto, S.-i. Iai, and N. Matsuoka}, When are the rings $I:I$ Gorenstein?, {\em Comm. Algebra} (to appear). 

\bibitem{EGI}
{\sc N. Endo, S. Goto, and R. Isobe}, Almost Gorenstein rings arising from fiber products, {\em Canad. Math. Bull.}, {\bf 64} (2021), no.2, 383--400.

\bibitem{GIW}
{\sc S. Goto, S.-i. Iai, and K.-i. Watanabe}, Good ideals in Gorenstein local rings, {\em Tran. Amer. Math. Soc.}, {\bf 353} (2000), 2309--2346.

\bibitem{GIK}
{\sc S. Goto, R. Isobe, and S. Kumashiro}, Correspondence between trace ideals and birational extensions with application to the analysis of the Gorenstein property of rings, {\em J. Pure and Appl. Algebra}, {\bf 224} (2020), no. 2, 747--767. 

\bibitem{GIT}
{\sc S. Goto, R. Isobe, and N. Taniguchi},  Ulrich ideals and 2-AGL rings, {\em J. Algebra}, {\bf 555} (2020), 96--130.

\bibitem{GK}
{\sc S. Goto and S. Kumashiro}, On generalized Gorenstein local rings, arXiv:2212.12762. 

\bibitem{GOTWY}
{\sc S. Goto, K. Ozeki, R. Takahashi, K.-i. Watanabe, and K.-i. Yoshida}, Ulrich ideals and modules, {\em Math. Proc. Cambridge Philos. Soc.}, {\bf 156} (2014), no.1, 137--166.


\bibitem{GMP}
{\sc S. Goto, N. Matsuoka, and T. T. Phuong},  Almost Gorenstein rings, {\em J. Algebra}, {\bf 379} (2013), 355--381.

\bibitem{GTT}
{\sc S. Goto, R. Takahashi, and N. Taniguchi}, Almost Gorenstein rings -towards a theory of higher dimension, {\em J. Pure Appl. Algebra}, {\bf 219} (2015), no. 7, 2666--2712.

\bibitem{GTT2}
{\sc S. Goto, R. Takahashi, and N. Taniguchi}, Ulrich ideals and almost Gorenstein rings, {\em Proc. Amer. Math. Soc.}, {\bf 144} (2016), 2811--2823.

\bibitem{GW}
{\sc S. Goto and K. Watanabe}, On graded rings I, {\em J. Math. Soc. Japan}, {\bf 30} (1978), no. 2, 179--213.

\bibitem{Hashimoto}
{\sc M. Hashimoto}, $F$-pure homomorphisms, strong $F$-regularity, and $F$-injectivity, {\em Comm. Algebra}, {\bf 38} (2010), no. 12, 4569--4596.

\bibitem{HK}
{\sc J. Herzog and E. Kunz}, Der kanonische Modul eines Cohen-Macaulay-Rings, Lecture Notes in Mathematics, {\bf 238}, {\em Springer-Verlag, Berlin-New York}, 1971.

\bibitem{HHS}
{\sc J. Herzog, T. Hibi, and D. I. Stamate}, The trace of the canonical module, {\em Israel J. Math.}, {\bf 233} (2019), 133--165. 

\bibitem{HR}
{\sc J. Herzog and M. Rahimbeigi}, On the set of trace ideals of a Noetherian ring, {\em Beitr. Algebra Geom.} (to appear). 


\bibitem{IK}
{\sc R. Isobe and S. Kumashiro}, Reflexive modules over Arf local rings, arXiv:2105.07184v1. 

\bibitem{IW}
{\sc O. Iyama and M. Wemyss}, The classification of special Cohen-Macaulay modules, {\em Math. Z.}, {\bf 265} (2010), no. 1, 41--83. 


\bibitem{K}
{\sc T. Kobayashi}, Syzygies of Cohen-Macaulay modules over one dimensional Cohen-Macaulay local rings, {\em Algebr. Represent. Theory}, {\bf 25} (2022), 1061--1070.

\bibitem{KT}
{\sc T. Kobayashi and R. Takahashi}, Rings whose ideals are isomorphic to trace ideals, {\em Math. Nachr.}, {\bf 292} (2019), no. 10, 2252--2261.


\bibitem{L}
{\sc H. Lindo}, Trace ideals and centers of endomorphism rings of modules over commutative rings, {\em J. Algebra}, {\bf 482} (2017), 102--130.

\bibitem{L2}
{\sc H. Lindo}, Self-injective commutative rings have no nontrivial rigid ideals, arXiv:1710.01793v2.

\bibitem{LP}
{\sc H. Lindo and N. Pande}, Trace ideals and the Gorenstein property, {\em Comm. Algebra}, {\bf 50} (2022), no. 10, 4116--4121.

\bibitem{Lipman}
{\sc J. Lipman}, Stable ideals and Arf rings, {\em Amer. J. Math.}, {\bf 93} (1971), 649--685. 

\bibitem{LW}
{\sc G. J. Leuschke and R. Wiegand}, Cohen-Macaulay Representations, Mathematical Surveys and Monographs, {\bf 181} (2012), {\em Amer. Math. Soc., Providence, RI}, 2012.

\bibitem{V1}
{\sc W. V. Vasconcelos}, Reflexive modules over Gorenstein rings, {\em Proc. Amer. Math. Soc.},  {\bf 19} (1968), 1349--1355.

\bibitem{Wunram}
{\sc J. Wunram}, Reflexive modules on quotient surface singularities, {\em Math. Ann.}, {\bf 279} (1988), no. 4, 583--598.

\end{thebibliography}
\end{document}